\documentclass{amsart}
\usepackage{lipsum}
\usepackage{amsfonts}
\usepackage{graphicx}
\usepackage{epstopdf}
\usepackage{graphicx,amssymb,algorithm,algorithmic,tikz,hyperref}
\usepackage{amsopn}
\newcommand{\vect}{\boldsymbol}
\renewcommand{\vec}{\boldsymbol}

\DeclareMathOperator{\arccot}{arccot}
\DeclareMathOperator{\Imag}{Im}
\DeclareMathOperator{\Real}{Re}

\usepackage{amsthm}
\newtheorem{lemma}{Lemma}
\newtheorem{theorem}{Theorem}
\theoremstyle{remark}
\newtheorem{remark}{Remark}

\usepackage{xcolor}

\title[Faber polynomials in a deltoid]{Faber polynomials in a deltoid region and power iteration momentum methods}

\author{Peter Cowal}
\email{cowalp@oregonstate.edu}
\address{Department of Mathematics, Oregon State University, Corvallis, OR}
\author{Nicholas F. Marshall}
\email{marsnich@oregonstate.edu}
\address{Department of Mathematics, Oregon State University, Corvallis, OR}
\author{Sara Pollock}
\email{s.pollock@ufl.edu}
\address{Department of Mathematics, University of Florida, Gainesville, FL}
\thanks{The third author's work was partially supported by NSF grant DMS 2045059 (CAREER)}


\begin{document}

\begin{abstract}
We consider a region in the complex plane enclosed by a deltoid curve inscribed in the unit circle, and define a family of polynomials $P_n$ that satisfy the same recurrence relation as the Faber polynomials for this region. We use this family of polynomials to give a constructive proof that $z^n$ is approximately a polynomial of degree $\sim\sqrt{n}$ within the deltoid region. Moreover, we show that $|P_n| \le 1$ in this deltoid region, and that, if $|z| = 1+\varepsilon$, then the magnitude $|P_n(z)|$ is at least $\frac{1}{3}(1+\sqrt{\varepsilon})^n$, for all $\varepsilon > 0$. We illustrate our polynomial approximation theory with an application to iterative linear algebra. In particular, we construct a higher-order momentum-based method that accelerates the power iteration for certain matrices with complex eigenvalues.  We show how the method can be run dynamically when the two dominant eigenvalues are real and positive.
\end{abstract}
\maketitle

\section{Introduction} \label{intro}
In computational mathematics and especially numerical linear algebra, many fast algorithms are based on the idea that $x^n$ is approximately a polynomial of degree $\sim \sqrt{n}$, which leads to a square root speed up over direct methods. 
Examples of such fast algorithms include Chebyshev iteration \cite{golub1961chebyshev}, Conjugate gradient\cite{hestenes1952methods}, GMRES \cite{saad1986gmres},  Lanczos algorithm \cite{lanczos1950iteration}, and Arnoldi iteration \cite{golub2006arnoldi}; for connections to quadrature, also see \cite{trefethen2022exactness}. 
On the interval $[-1,1]$, an explicit statement of this idea first appeared in the 1976 paper of Newman and Rivlin \cite{newman1976approximation}  who proved that $x^n$ can be uniformly approximated using a polynomial of degree $\sim \sqrt{n}$ and that an accurate approximation cannot be achieved by a lower degree polynomial.  
A more precise result was obtained by \cite{sachdeva2014faster},  (see Theorem \ref{chebxnfast} below). A related property of Chebyshev polynomials is that the magnitude of the $n$-th degree Chebyshev polynomial $|T_n(x)|$ is at least $\frac{1}{2} (1+ \sqrt{2 \varepsilon})^n$ for $|x| = 1 + \varepsilon$ with $\varepsilon > 0$, which again is a factor of a square root more than we might expect.
In this paper, we use a family of polynomials related to Faber polynomials to generalize these efficient approximation and rapid growth properties to a deltoid region in the complex plane.

The motivation for our approach is twofold. First, we are motivated by \cite{austin2024dynamically,xu2018accelerated} that consider accelerating the power method from an optimization perspective inspired by Polyak's heavy ball momentum \cite{polyak1964some}, where \cite{xu2018accelerated} establishes a connection between momentum and Chebyshev polynomials, and \cite{austin2024dynamically} uses the dynamic extrapolation techniques of  \cite{nigam2022simple,pollock2021extrapolating} to determine the optimal momentum.  
For a comparison of this dynamic momentum approach to the Lanczos method for symmetric problems, see \cite{barletta2025momentum}.
Second, we are motivated by generalizing the probabilistic concentration proof of Sachdeva and Vishnoi \cite{sachdeva2014faster}, see Theorem \ref{chebxnfast} below, to a region of the complex plane using polynomials defined by higher-order recurrence formulas.

In the following, we briefly introduce Faber polynomials; for more background, see \cite{suetin1998series}. Let $\Omega$ be a compact subset of the complex plane that is not a single point, whose complement $\overline{\mathbb{C}} \setminus \Omega$ is simply connected in the extended complex plane $\overline{\mathbb{C}}$. By the Riemann mapping theorem there is a conformal map $\psi$ from the exterior of the unit disk $\{z \in \overline{\mathbb{C}} : |z| > 1\}$ to $\overline{\mathbb{C}} \setminus \Omega$. To make the choice unique, assume $\psi(\infty) = \infty$ and $\psi'(\infty) > 0$. The Faber polynomials $F_n(z)$, $n \ge 0$ for $\Omega$ are defined by the generating function
\begin{equation} \label{generate}
\frac{\psi'(w)}{\psi(w) - z} = \sum_{k=0}^\infty F_k(z) w^{-k-1} \quad \text{for} \quad |w| > 1 \quad \text{and} \quad z \in \Omega,
\end{equation}
see Section \ref{prelmin} below for a concrete example.
The Faber polynomials $F_n(z)$ are a natural polynomial basis in $\Omega$ in the sense that any analytic function in $\Omega$ can be expanded in a convergent series of Faber polynomials; see \cite[Section 2]{curtiss1971faber}.  When the map $\psi$ is rational, the Faber polynomials satisfy a recurrence relation, see \cite[Chapter 2 Section 2]{suetin1998series}. In this paper, we consider a family of polynomials $P_n$ defined in \eqref{newpoly} that satisfy the same recurrence relation as Faber polynomials $F_n$ on our domain of interest, but with initial conditions suited for our application, analogous to the relation between Chebyshev polynomials of the first and second kind.

Faber polynomials have been used in a variety of applications in iterative linear algebra, including solving linear systems \cite{eiermann1989semiiterative,eiermann1989hybrid,nachtigal1992hybrid,niethammer1983analysis,starke1993hybrid}, computing functions of matrices such as the matrix exponential \cite{beckermann2009error,moret2001computation,moret2001interpolatory},
and finding eigenvalues \cite{heuveline1996arnoldi,pozza2019inexact}.
These methods estimate or assume a region of the complex plane that contains the eigenvalues of a given matrix, and then use Faber polynomials on the domain to accelerate an iterative scheme.
Of these works, the one most related to the current paper is \cite{eiermann1989hybrid}, which considers hybrid iterative methods for solving linear systems using Faber polynomials,
including methods based on $k$-step recurrence formulas where hypocycloid regions emerge; however, the analysis focuses on aspects different from the current paper. 

In this paper, we consider the behavior of Faber polynomials in the interior and exterior of a deltoid region, that is, a three-cusp hypocycloid.
Hypocycloid regions are natural to consider when studying Faber polynomials, since on these domains the generating function \eqref{generate} yields easily computable recurrence formulas for the Faber polynomials, see \cite[Section 2]{he1994zeros}. For hypocycloid regions, the location, density, and asymptotic behavior of the zeros of the Faber polynomials are understood  \cite{eiermann1993zeros,he1994zeros}. Moreover, on these domains, Faber polynomials are orthogonal with respect to measures on a set consisting of a union of rays emanating from the origin   \cite{aptekarev2009higher}.

\subsection{Main contributions}
This work relates to but takes a different approach from previous numerical-focused investigations involving Faber polynomials that consider general domains in the complex plane, or large families of domains. In this paper, we restrict our attention to a deltoid region of the complex plane, see Figure \ref{fig:Pn-stability}, and focus on developing a precise theoretical understanding.  
The main contributions of this paper are as follows: 
\begin{enumerate}
\item We prove that a family of polynomials $P_n$, related to the Faber polynomials on a deltoid region of the complex plane, has boundedness and rapid growth properties (see Theorem \ref{newpolyboundedgrow}), which are similar to Chebyshev polynomials.
\item We prove that  $z^n$ is approximately a polynomial of degree $\sim \sqrt{n}$ in a deltoid region of the complex plane
(see Theorem \ref{polynomialapproxcomplex}), which generalizes the results of \cite{newman1976approximation,sachdeva2014faster}, which consider the interval $[-1,1]$.
\item We introduce higher-order static and dynamic momentum algorithms for finding eigenvalues (see Algorithms \ref{algnew} and \ref{alg:dymo2} and Theorem \ref{algoworks}), which generalize
\cite{austin2024dynamically,xu2018accelerated} to certain non-symmetric matrices.
\end{enumerate}
While our application to eigenvalue momentum algorithms is based on Theorem \ref{newpolyboundedgrow}, the result of Theorem \ref{polynomialapproxcomplex} also has immediate applications as discussed in Section \ref{discussion}.

\subsection{Preliminaries} \label{prelmin}
The Chebyshev polynomials $T_n$ (of the first kind), defined by $T_0(x) = 1$, $T_1(x) = x$, and $T_{n+1}(x) = 2 x T_{n}(x) - T_{n-1}(x),$ for $n \ge 1$, are bounded in $[-1,1]$ and grow rapidly outside of $[-1,1]$. More precisely, the following result holds.

\begin{lemma} \label{chebgrowrate}
For all $n \ge 0$, the Chebyshev polynomials satisfy
$|T_n(x)| \le 1$ for $x \in [-1,1]$, and for all $\varepsilon > 0$
$$
|T_n(x)| \ge \frac{1}{2} (1+\sqrt{2\varepsilon})^n \quad \text{for} \quad x \in \mathbb{R} : |x| = 1+\varepsilon.
$$
\end{lemma}

Lemma \ref{chebgrowrate} is classical and follows, for example, from \cite[Proposition 2.5]{sachdeva2014faster}. Note that this growth rate is a square root gain over the monomial $x^n$ that grows like $(1+\varepsilon)^n$ when $x = 1+\varepsilon$. A similar square root gain over monomials is seen in the following approximation result, which, informally speaking, says that $x^n$ is approximately a polynomial of degree $\sim \sqrt{n}$ on $[-1,1]$. Throughout the paper, we write $\mathbb{P}$ to denote probability and $\mathbb{E}$ to denote expectation. 

\begin{theorem}[Sachdeva, Vishnoi,  \cite{sachdeva2014faster}]  \label{chebxnfast}
Fix $n \in \mathbb{N}$ and $t > 0$. Then,
\begin{equation} \label{monomialeq}
\left|x^n - \sum_{k=0}^{\lfloor t \sqrt{n} \rfloor} \alpha_k T_k(x) \right| \le 2 e^{-t^2/2}, \qquad \text{for} \quad x \in [-1,1],
\end{equation}
where the coefficients $\alpha_k = \mathbb{P}(|X_1+\cdots + X_n| = k)$,
where $X_1,\ldots,X_n$ are i.i.d. random variables satisfying
$\mathbb{P}(X_j = -1) = \mathbb{P}(X_j=1) =1/2$ for $j \in \{1,\ldots,n\}$.
\end{theorem}

While the Chebyshev polynomials satisfy $|T_n| \le 1$ on $[-1,1]$, for all other points in the complex plane, they grow exponentially with $n$. The Faber polynomials generalize the Chebyshev polynomials to regions of the complex plane, see  \cite[Chapter 1.1, Example 2]{suetin1998series}. In this paper, we restrict our attention to the deltoid region of the complex plane in Figure \ref{fig:Pn-stability}, which we denote by $\Gamma$. 

In the following, we sketch the derivation of Faber polynomials for $\Gamma$
using the definition of Faber polynomials in Section \ref{intro}. By considering the action of $\psi$ on the boundary of $\Gamma$, see \eqref{gamma}, observe that 
\begin{equation} \label{psifordeltoid}
\psi(w) = \frac{2}{3} w + \frac{1}{3} w^{-2}
\end{equation}
is a conformal map from the exterior of the unit disk onto $\overline{\mathbb{C}} \setminus \Gamma$ such that $\psi(\infty) = \infty$ and $\psi'(\infty) > 0$. Substituting \eqref{psifordeltoid} into \eqref{generate} and computing the first few terms of the series expansion gives 
$$
\frac{1 - w^{-3}}{w + \frac{1}{2} w^{-2} - \frac{3}{2} z} = 
w^{-1} + \frac{3}{2} z w^{-2} + \frac{9}{4} z^2 w^{-3} + \frac{27 z^3 - 12}{8}  w^{-4} + \frac{81 z^4- 48 z}{16} w^{-5} + \cdots,
$$
where the coefficient of $w^{-k-1}$ on the right-hand side is the $k$-th Faber polynomial $F_k(z)$ for $\Gamma$.
For $n \ge 3$, the Faber polynomials on $\Gamma$ satisfy the recurrence $F_{n+1}(z) = \frac{3}{2}z F_n(z) - \frac{1}{2}F_{n-2}(z)$, see \cite[Section 3(c)]{curtiss1971faber}; we emphasize that this recurrence does not hold for $n=2$. In this paper, we consider a family of polynomials $P_n$ that we define in \eqref{newpoly} below, which satisfy the same recurrence relation as the Faber polynomials $F_n$, but  
have different initial conditions adapted to our purpose. The polynomials $P_n$ and $F_n$ share many properties; in particular, the zeros of $F_n$ and $P_n$ are contained on a union of rays emanating from the origin, see \cite{he1994zeros} and Figure \ref{fig:Pn-plots}.

\subsection{Main analytic results} \label{mainresults}
Define the family of polynomials $P_n(z), ~n \ge 0,$ by
$P_0(z) = 1$, $P_1(z) = z$, $P_2(z) = z^2$, and
\begin{equation} \label{newpoly}
P_{n+1}(z) = \tfrac{3}{2}z P_n(z) - \tfrac{1}{2}P_{n-2}(z), \quad \text{for} \quad \quad n \ge 2.
\end{equation}
Consider the closed curve $\gamma \subset \mathbb{C}$ parameterized by
\begin{equation} \label{gamma}
\gamma(t) = \frac{2}{3}e^{i t} + \frac{1}{3}e^{-2i t}
\quad \text{for $t \in [0,2\pi]$},
\end{equation}
and let $\Gamma$ be the closure of the region enclosed by $\gamma$, see Figure \ref{fig:Pn-stability}.

\begin{figure}[h!]
    \centering
    \begin{tikzpicture}[scale=2]

\fill[gray!20]
  plot[smooth,domain=0:360,samples=200]
    ({2/3*cos(\x) + 1/3*cos(-2*\x)}, {2/3*sin(\x) + 1/3*sin(-2*\x)})
  -- cycle; 

\draw (-1.25, 0) -- (1.25, 0);
\draw (0, -1.25) -- (0, 1.25);

\draw (1, 0) node[anchor=south west]{$1$};
\draw (0, 1) node[anchor=south west]{$i$};

\draw [dashed] (0,0) circle (1);

\draw[dashed] (0, 0) circle (1/3);

\draw (1/6, 1.732/6) node[anchor = south west]{$|z| = \tfrac{1}{3}$};

\draw (9/12, .6) node[anchor = south west]{$|z| = 1$};

\draw (0, 0) node[anchor=south east]{$\Gamma$};

\draw (-9/24, 1.732/6) node[anchor =east]{$\gamma$};

\draw [ultra thick, domain = 0:360, samples = 60]
    plot[smooth] ({2/3*cos(\x) + 1/3*cos(-2*\x)}, {2/3*sin(\x) + 1/3*sin(-2*\x)});

\end{tikzpicture}
    \caption{The curve $\gamma$ defined by \eqref{gamma}, called a deltoid, intersects the circle $|z|=1$ at three points
    and circle $|z|=1/3$ at 
    three points. The closure of the region enclosed by $\gamma$ is denoted by $\Gamma$.
    }
    \label{fig:Pn-stability}
\end{figure}

Our first main result says $P_n$ is bounded in $\Gamma$
and grows at a similar rate to Chebyshev polynomials outside the unit disk.

\begin{theorem} 
\label{newpolyboundedgrow}
For all $n \ge 0$, we have
$$
|P_n(z)| \le 1, \quad \text{for all} \quad z \in \Gamma.
$$
 Moreover, 
for all $n \ge 0$ and $\varepsilon > 0$,
$$
|P_n(z)| \ge  \frac{1}{3} (1+\sqrt{\varepsilon})^n \quad \text{for} \quad z \in \mathbb{C} : |z| = 1+\varepsilon.
$$
\end{theorem}

The proof of Theorem \ref{newpolyboundedgrow} is given in Section
\ref{proofthm1}. 
 Our second main result says that $z^n$ can be approximated by a polynomial of degree $\sim \sqrt{n}$ in $\Gamma$ in the basis of polynomials $P_k$ using nonnegative bounded coefficients.

\begin{theorem} \label{polynomialapproxcomplex} For any fixed $n \ge 0$ and $t > 0$ we have
$$
\left| z^n - \sum_{k=0}^{\lfloor t \sqrt{n} \rfloor} \beta_k P_k(z) \right| \le 2 e^{-t^2/7},
\quad \text{for all} \quad z \in \Gamma,
$$
where the coefficients $\beta_k$ are random walk probabilities  $\beta_k = \mathbb{P}(|Y_n| = k),$
where $Y_n$ is the Markov random walk defined in \eqref{keyrandomwalk}.
\end{theorem}

The proof of Theorem \ref{polynomialapproxcomplex} is given in Section \ref{proofthm2}. 
In Figure \ref{fig:Pn-plots} we plot the magnitude of $P_n(z)$ on a square centered at the origin, highlighting the locations of the zeros of $P_n$.  We also note that as $n$ grows, the region where $|P_n| \leq 1$ increasingly resembles the deltoid region illustrated in Figure \ref{fig:Pn-stability}.
\begin{figure}[h!]
    \centering
    \includegraphics[width=.8\linewidth]{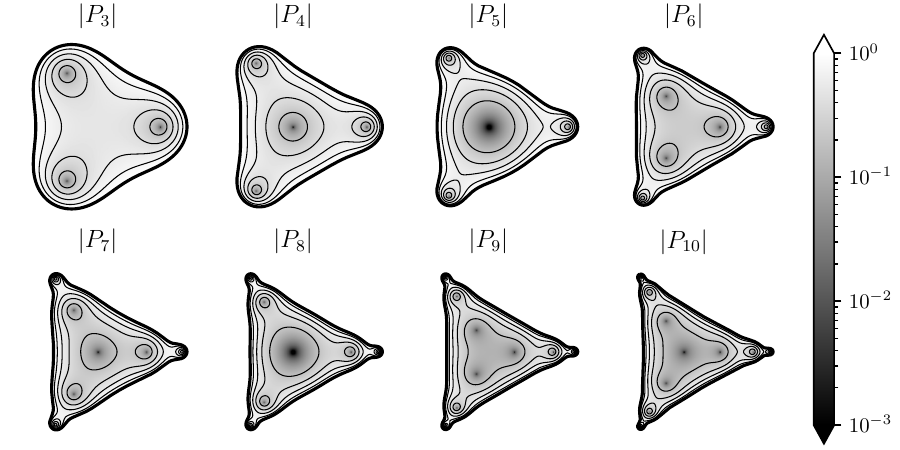}
    \caption{Plots of $|P_n(z)|$ on 
    $\{z \in \mathbb{C}: |\Real(z)| \leq 1, |\Imag(z)| \leq 1 \}$, where points such that $|P_n(z)| \leq 10^{-3}$ are shown in black, and points such that $|P_n(z)| \geq 1$ are shown in white.   The contours indicate $|P_n(z)| \in \{0.2, 0.4, 0.6, 0.8, 1\}$.}
    \label{fig:Pn-plots}
\end{figure}

\section{Application to the power iteration with momentum}\label{sec:power}
In this section, we describe a direct application of Theorem
\ref{newpolyboundedgrow} to accelerating the power iteration. 
In preparation for using the polynomials $P_n(z)$ for a modified power method, we scale and dilate these polynomials such that they satisfy a recurrence formula involving multiplying the $n$-th polynomial by $z$ instead of $\frac{3}{2}z$ as in \eqref{newpoly}. This modification will make it easier to compare the resulting algorithm to the power method. In particular, we define
$$
\tilde{P}_n(z) = \left(\frac{2\lambda}{3} \right)^n P_n(z/\lambda), \quad \text{for} \quad n \ge 0,
$$
such that $\tilde{P}_{0}(z) = 1$, $\tilde{P}_1(z) = \frac{2}{3} z$, and $\tilde{P}_2(z) = \frac{4}{9} z^2$, and
\begin{equation} \label{scaledilatepn}
\tilde{P}_{n+1}(z) = z \tilde{P}_n(z) - \tfrac{4 \lambda^3}{27}\tilde{P}_{n-2}(z),
\quad \text{for}  \quad n \ge 2.
\end{equation}
In particular, take note of the constant $4\lambda^3/27$, which will make an appearance below.

\subsection{Notation} \label{algnotation}
The following notation and assumptions are used throughout this section. Let $\vect{A} \in \mathbb{C}^{n \times n}$ be a diagonalizable matrix 
(this assumption could be generalized) 
with eigenvalues $\lambda_1,\ldots,\lambda_n$. Assume $\vec{A}$ has a unique eigenvalue of largest magnitude and its other eigenvalues are sorted in descending order by magnitude
$$
|\lambda_1| > |\lambda_2| \ge \cdots \ge |\lambda_n|.
$$
Let $\vec{\phi}_1,\ldots,\vec{\phi}_n$ be the corresonding normalized eigenvectors.
Given $\vect{x},\vect{y} \in \mathbb{C}^n$, let $\langle \vec{x}, \vec{y} \rangle = \sum_{i=1}^n \vec{x}_i \overline{\vec{y}}_i$ and $\|\vect{x}\|_2 = \sqrt{\langle \vect{x}, \vec{x} \rangle}$.

\subsection{Algorithms} \label{secalgorithms}
First, in Algorithm \ref{alg1}, we precisely state the power method, which we use for comparison and to initialize the proposed methods.  
\begin{algorithm}[h!] 
\caption{Power method}
\label{alg1}
\begin{algorithmic}[1]
\REQUIRE Matrix $\vect{A} \in \mathbb{R}^{n \times n}$, vector $\vect{u}_0 \in \mathbb{R}^n$,  $N \in \mathbb{N}$
\STATE Set $h_0 = \|\vect{u}_0\|_2$ and $\vect{x}_0 = h_{0}^{-1}\vect{u}_{0}$
\FOR{$k = 0,\dots,N-1$}
    \STATE Set $\vect{u}_{k+1} = \vect{A} \vect{x}_k$
    \STATE Set $h_{k+1} = \|\vect{u}_{k+1}\|_2$ and $\vect{x}_{k+1} = h_{k+1}^{-1}\vect{u}_{k+1}$
    \STATE Set $\nu_{k} = \langle \vect{u}_{k+1}, \vec{x}_k \rangle$ and $d_{k} = \| \vect{u}_{k+1} - \nu_{k}\vect{x}_k\|_2$ \hfill $\triangleright$ \emph{defined for later use}
\ENDFOR
\RETURN $\vec{x}_N$
\end{algorithmic}
\end{algorithm}
Next, in Algorithm \ref{algnew}, we present the Deltoid Momentum Power Method, a modification of the power method inspired by Theorem \ref{newpolyboundedgrow}.
In \cite{xu2018accelerated}, a momentum accelerated power iteration for problems with real eigenvalues is developed based on the recursion for Chebyshev polynomials. As shown in \cite{austin2024dynamically,xu2018accelerated}, this iteration given by $\vect{u}_{k+1} = \vec{A}\vect{x}_k - \beta/h_k \vect{x}_{k-1}$ with normalization $\vect{x}_k = \vect{u}_k/h_k$, $h_k =\|\vect{u}_k\|_2$, can be written in terms of a power iteration applied to an augmented matrix. 
In \cite{xu2018accelerated} it is shown that the optimal parameter is $\beta = \lambda_2^2/4$. The main contribution of \cite{austin2024dynamically} was showing the optimal (static) method could be efficiently approximated by a dynamically set parameter based on a posteriori information, 
without a priori knowledge of $\lambda_2$.

Similarly, Algorithm \ref{algnew}, which is based on a Faber polynomial recursion, can be expressed in terms of a power iteration applied to an augmented matrix.
Observe that writing $\vect{x}_n = \tilde P_n(\vec{A})\vect{x}_0$, the rescaled polynomial iteration \eqref{scaledilatepn} with $z$ replaced by $\vec{A}$,  defines, up to normalization, a power-like iteration $\vect{x}_{n+1} = \vec{A} \vect{x}_n - \beta \vect{x}_{n-2}$, 
where, as in \eqref{scaledilatepn}, the optimal choice for parameter $\beta$ is
$\beta = 4 \lambda_2^3/27$.
An appropriate normalization follows naturally from formulating the accelerated iteration as a power iteration applied to an augmented matrix $\vec{A}_\beta$ given by
\[ \vect{A}_\beta = \begin{pmatrix}
\vect{A} & \vect{0} & -\beta \vect{I}\\
\vect{I} & \vect{0} & \vect{0} \\
\vect{0} & \vect{I} & \vect{0} \end{pmatrix}.
\]
The accelerated iteration 
$\vect{u}_{k+1} = \vect{A} \vect{x}_{k} - (\beta/(h_k h_{k-1})) \vect{x}_{k-2}$ 
in line 3 of Algorithm \ref{algnew} 
results from applying the power iteration to the augmented matrix $\vect{A}_\beta$ by
$\tilde{\vect{u}}_{k+1} = \vect{A}_\beta \tilde{\vect{x}}_{k}$,  for $\tilde{\vect{x}}_k = \tilde{\vect{u}}_k/h_k$ 
and $h_k = \|\vect{u}_k\|_2$ where $\vect{u}_k$
 is the $n$-dimensional vector consisting of the first $n$ entries of $\tilde{\vec{u}}_k$. 
 The details can be found by following the arguments in \cite{austin2024dynamically}.
 Algorithm \ref{algnew} as it is presented is efficient but not practical: to be efficient, the algorithm requires a good choice of the parameter $\beta$, which, as will be seen in Theorem \ref{algoworks}, requires knowledge of $\lambda_2$, which is in general a priori unknown. 
 A dynamic method to approximate Algorithm \ref{algnew}, analogous to the approach of \cite{austin2024dynamically}, is presented in Subsection \ref{subsec:dynamic}.

\begin{algorithm}[h!] 
\caption{Deltoid Momentum Power Method}
\label{algnew}
\begin{algorithmic}[1]
\REQUIRE 
Matrix $\vect{A} \in \mathbb{R}^{n \times n}$,
vector $\vect{v}_0 \in \mathbb{R}^n$,  parameter $\beta  \in \mathbb{C}$, $N \in \mathbb{N}$
\STATE Do two iterations of Algorithm \ref{alg1} with inputs $\frac{2}{3} \vect{A}$ and $\vec{v}_0$
\FOR{$k = 2, \dots, N - 1$} 
\STATE Set $\vect{v}_{k+1} = \vect{A} \vect{x}_k$
\STATE{Set $\vect{u}_{k+1} = \vect{v}_{k+1} - (\beta /(h_{k} h_{k-1})) \vect{x}_{k-2}$}
\STATE{Set $h_{k+1} = \|\vect{u}_{k+1}\|_2$ and $\vect{x}_{k+1} = h_{k+1}^{-1} \vect{u}_{k+1}$}
\ENDFOR
\RETURN  $\vec{x}_N$ 
\end{algorithmic}
\end{algorithm}

\subsection{Theoretical guarantees} \label{secguarantee}

\begin{theorem} \label{algoworks} 
In addition to the assumptions of Section \ref{algnotation}, assume that $\lambda_* \in \mathbb{C}$ is given such that
$|\lambda_*| < |\lambda_1|$ and
\begin{equation} \label{eqdeltoidcondition}
\lambda_2,\ldots,\lambda_n \in \lambda_* \Gamma  = \{ \lambda_* z \in \mathbb{C} : z \in \Gamma\}.
\end{equation}
Let $\vec{v}_0 \in \mathbb{R}^n$ be an initial vector with eigenbasis expansion $\vec{v}_0 = \sum_{j=1}^n a_j \vect{\phi}_j$. 
Run Algorithm \ref{algnew} with inputs $\vect{A}$, $\vect{v}_0$, $\beta = 4\lambda_*^3/27$, and $N$. Then, the output $\vect{x}_N$ satisfies
$$
\min_{\theta \in [0,2\pi)} \| \vec{x}_N  - e^{i \theta}\vec{\phi}_1\|_2 = \mathcal{O} \left( \left(1+\sqrt{\left|\frac{\lambda_1}{\lambda_*}\right| - 1}\right)^{-N} \right),
$$
as $N \to \infty$, where the implied constant in the big-$\mathcal{O}$ only depends on $a_1, \dots, a_n$. 
\end{theorem}

The proof of Theorem \ref{algoworks} is given in Section \ref{proofalgoworks}. 

\begin{remark}[Optimizing the error bound]  \label{rmkconv}
Under the assumption that \eqref{eqdeltoidcondition} holds with $\lambda_*=\lambda_2$, the error bound is optimized when $\lambda_* = \lambda_2$. Indeed, if $|\lambda_*| < |\lambda_2|$, then the condition $\lambda_2,\ldots,\lambda_n \in \lambda_* \Gamma$ cannot hold since $\Gamma$ is contained in the unit disk. Assume $\lambda_* = \lambda_2$ and 
set $\varepsilon = |\lambda_1/\lambda_2| - 1$. Then, for small $\varepsilon$,
in terms of the convergence rate, Theorem \ref{algoworks} says that the error for the estimate of $\vec{\phi}_1$ decays like 
$$
(1+\sqrt{\varepsilon})^{-N} \approx e^{-N \sqrt{\varepsilon}},
\quad \text{compared to} \quad
\left|\frac{\lambda_2}{\lambda_1} \right|^N = (1+\varepsilon)^{-N} \approx e^{-N \varepsilon},
$$
for the power method. Note that $\varepsilon$ can be interpreted as the size of the spectral gap $|\lambda_1| - |\lambda_2|$ relative to $|\lambda_2|$. Thus, in the critical case when the gap is small, 
we achieve a square root gain over the power method with respect to the relative spectral gap.
\end{remark}

\subsection{Dynamic parameter assignment}\label{subsec:dynamic}
In this section, we assume that \eqref{eqdeltoidcondition} holds with $\lambda_*=\lambda_2$ such that $\lambda_* = \lambda_2$ optimizes the error bound, see Remark \ref{rmkconv}.
We will consider an efficient and practical approximation of Algorithm \ref{algnew} in which the optimal parameter $\beta = 4\lambda_2^3/27$ is approximated by a sequence $\beta_k$ that does not rely on a priori knowledge of $\lambda_2$. The main idea follows that in \cite{austin2024dynamically}.  For the remainder of this section, in addition to the assumptions of Section \ref{algnotation}, we assume
$\lambda_1, \lambda_2 \in \mathbb{R}$, $\lambda_1 > \lambda_2 > 0$.

In light of Remark \ref{rmkconv}, the convergence rate, given by
the ratio of consecutive normed errors for Algorithm \ref{algnew} is asymptotically governed by
$e^{-(N+1)\sqrt{\varepsilon}}/e^{-N\sqrt{\varepsilon}} = e^{-\sqrt{\varepsilon}}$. Writing $\varepsilon$ in terms of $r = |\lambda_2/\lambda_1|$ we have $\varepsilon = 1/r -1$, yielding a convergence rate
$\rho(r) = e^{-\sqrt{1/r -1}}$,
a bijective map $(0,1) \rightarrow (0,1)$. Solving for $r$ in terms of $\rho$ in that interval yields
\begin{align}\label{eqn:rofrho}
r(\rho) = \frac{1}{(\log \rho)^2 + 1}.
\end{align}
In the dynamic Algorithm \ref{alg:dymo2}, we monitor the convergence rate $\rho_k$ given by the ratio of normed residual vectors (denoted $d_k$), and use \eqref{eqn:rofrho} to approximate $r$ by $r_{k+1}$, which is used together with the Rayleigh quotient (denoted $\nu_k$) to produce $\beta_{k}$.
The additional assumption that the first two eigenvalues are real and positive comes from the use of the ratio of normed residuals to approximate $\rho$, which produce a real and positive approximation to $r$ via $\eqref{eqn:rofrho}$, by which the approximation to $\lambda_2$ necessarily agrees with the sign of the Rayleigh quotient. 

\begin{algorithm}
\caption{Dynamic Deltoid Momentum Power Method}
\label{alg:dymo2}
\begin{algorithmic}[1]
\REQUIRE  Matrix $\vect{A} \in \mathbb{R}^{n \times n}$, vector $\vect{v}_0  \in \mathbb{R}^n$, $N \in \mathbb{N}$.
\STATE Do two iterations of Algorithm \ref{alg1} with inputs $\frac{2}{3} \vect{A}$ and $\vec{v}_0$
\FOR{$k =2,\ldots,N-1$} 
\STATE{Set $\vect{v}_{k+1} = \vect{A} \vect{x}_{k}$, $\nu_{k} = \langle \vect{v}_{k+1}, \vect{x}_{k}\rangle$ and $d_{k} = \|\vect{v}_{k+1} - \nu_{k} \vect{x}_{k}\|_2$}
\STATE{Set $\rho_{k-1} = \min\{d_{k}/d_{k-1},1\}$  and $r_k = 1/((\log \rho_{k-1})^2 + 1)$}
\STATE{Set $\beta_k = 4(\nu_{k} r_{k})^3/27$}
\STATE{Set $\vect{u}_{k+1} = \vect{v}_{k+1} - (\beta_k/(h_k h_{k-1})) \vect{x}_{k-2}$}
\STATE{Set $h_{k+1} = \|\vect{u}_{k+1}\|_2$ and $\vect{x}_{k+1} = h_{k+1}^{-1} \vect{u}_{k+1}$}
\ENDFOR
\RETURN $\vec{x}_N$
\end{algorithmic}
\end{algorithm}

While we do not formally establish the convergence rate of Algorithm \ref{alg:dymo2}, we do establish the essential mechanism by which Algorithm \ref{alg:dymo2} dynamically approximates Algorithm \ref{algnew}, in Lemma \ref{lem:contraction}
and Remark \ref{rmk:stabilitydynamic}.
The following lemma shows that $r(\rho)$ given by \eqref{eqn:rofrho} is a contraction. 

\begin{lemma} \label{lem:contraction}
Let $r(\rho)$ be defined by \eqref{eqn:rofrho}. Then, 
$$
|r(\rho) - r(\rho')| \le c | \rho - \rho'|, \quad \text{for} \quad \rho,\rho' \in (0.629,1),
$$
for an absolute constant $0 < c < 1$.
Moreover if $\rho_j$ is an increasing sequence that tends to $1$ with $0.629 \le \rho_j < 1$, then $r(\rho)$ is a contraction mapping in $(\rho_j,1)$ for a corresponding sequence of constants of $c_j$ that decrease monotonically to zero.
\end{lemma}

The proof of this lemma is based on showing that $r'(\rho)$ is decreasing on the interval $(e^{-1},1)$ and using the fact that $1 > r'(0.629) > r'(1) = 0$, where the value $0.629$ is a numerical upper bound to the solution to the equation $r'(\rho)= 1$ 
for $\rho \in (e^{-1},1)$.

\begin{proof}[Proof of Lemma \ref{lem:contraction}]
We start by computing the first two derivatives of $r$ with respect to $\rho$
$$
r'(\rho) = \frac{-2 \log \rho}{\rho(1 +(\log \rho)^2)^2}, \quad \text{and} \quad r''(\rho) = \frac{
2\big(1+\log \rho\big) \big(-1+(\log\rho) \big(2+\log \rho \big)))}{\rho^2(1+(\log \rho)^2)^3}.
$$
We claim that $r''(\rho) < 0$ for $\rho \in (e^{-1},1)$. Indeed, the denominator of the expression defining $r''$ is always positive, while in the numerator $(1+\log \rho)$ is positive and $(-1+(\log\rho) (2+\log \rho ))$ is negative when $\rho \in (e^{-1},1)$. Thus $r'(\rho)$ is decreasing on $(e^{-1},1)$.
Since $e^{-1} < 0.629$ and
$$
r'(0.629) \approx 0.998689 < 0.999 \quad \text{and} \quad r'(1) = 0,
$$
it follows that  $|r'(\rho)| \le 0.999$
for $\rho \in (0.629,1)$. Applying the mean value theorem shows that $r(\rho)$ is a contraction mapping on $(0.629,1)$ for constant $c = 0.999$. The second claim in the statement of the lemma follows from the fact that  $r'(\rho)$ is decreasing on $(0.629,1)$, and satisfies $r'(1) = 0$.
\end{proof}
We note that $\rho \in (0.629,1)$ corresponds to $r \in (0.823,1)$, or $\varepsilon \in $(0,0.215).
\begin{remark}[Stability of dynamic algorithm] \label{rmk:stabilitydynamic}
The results of Theorem \ref{algoworks} hold if $\beta_k = 4 \lambda_*^3/27$, with $\lambda_* \in [\lambda_2,\lambda_1)$. If $\lambda_* < \lambda_2$, it means the component of $\vec{x}_k$ along $\vec{\phi}_2$ will blow up with (but less than) the component along $\vec{\phi}_1$. Note that this agrees with the technique used in \cite{saad84} to filter out the smaller eigenmodes from the approximation.
With this in mind, if an initial $\beta_k$ is too close to zero, the detected convergence rate will be closer to one than the estimate used to form $\beta_k$. In this way, if the true $r = |\lambda_2/\lambda_1|$ is close enough to one, initial approximations to $\rho_{k}$ that lie outside of the contraction regime of Lemma \ref{lem:contraction} naturally map into that regime in future iterations. If $\vec{A}$ is not symmetric, we do not have a guarantee that the Rayleigh quotient $\nu_k$ will be less than $\lambda_1$, but we still find that the difference between the sequence $\beta_k$ and the optimal $\beta$ is dominated by the error in $\rho_k$ as an approximation to $\rho$. As our predicted convergence rate $\rho$ from Remark \ref{rmkconv} is asymptotically accurate, our sequence of approximations $\rho_k$, $r_k$, and $\beta_k$ can be expected to asymptotically approach their predicted values for $\lambda_2/\lambda_1 \approx 1$.
\end{remark}

\subsection{Numerical implementation} \label{numerics}
This section reports numerical results resulting from running the algorithms introduced in Subsections \ref{secalgorithms} and \ref{subsec:dynamic} on an illustrative toy example and a stationary distribution example. Code that reproduces the numerical experiments is available at: 
\begin{center}
\url{https://github.com/petercowal/higher-order-momentum-power-methods}
\end{center}

\subsubsection{Toy example} \label{toyexample}
We consider the toy problem of finding the dominant eigenvector of the matrix
\begin{equation}\label{eq:toy-problem}
\vect{A} = \begin{bmatrix}
    \frac{101}{100} & 0 & 0 & 0 \\ 0 & 1 & 0 & 0\\ 0 & 0 & 0 & -\frac{1}{3} \\ 0 & 0 & \frac{1}{3} & 0
\end{bmatrix}.
\end{equation}
The eigenvalues of $\vec{A}$ are $\lambda_1 =101/100$, $\lambda_2 = 1$, $\lambda_3=i/3$ and $\lambda_4 = -i/3$.
By construction, $\lambda_3,\lambda_4 \in \lambda_2 \Gamma$, see Figure \ref{fig:Pn-stability}. Therefore, if we set $\beta=4/27$ and run  Algorithm \ref{algnew}, then Theorem \ref{algoworks} implies the error between the output $\vec{y}_N$ and $\vec{\phi}_1$ is order $(10/11)^{N}$, while the error for the power method is order $(100/101)^{N}$.  We report numerical results for this example in Section \ref{numericalresults}.

\subsubsection{Circulant matrix example} \label{circulantexample}
We construct a second toy problem to demonstrate what happens when $\lambda_2, \dots, \lambda_n$ lie exactly on the deltoid $\gamma$.  Let $\vec{A} \in \mathbb{R}^{n \times n}$ be defined in terms of the circulant matrix $\vec{C} \in \mathbb{R}^{n-1 \times n-1}$ as follows:
$$\vec{A} = \begin{bmatrix} 1 + \varepsilon & 0 \\ 0 & \vec{C} \end{bmatrix} \quad \text{where} \quad \vec{C} = \begin{bmatrix}
0 & 2/3 &   &  1/3 & 0 \\
0 & \ddots & \ddots &    & 1/3 \\
1/3 & \ddots & \ddots& \ddots&  \\
 & \ddots & \ddots & \ddots & 2/3 \\
 2/3  & & 1/3 & 0 & 0\end{bmatrix}.$$
The eigenvalues of the matrix $\vec{C}$ lie exactly on $\gamma$; for details, see \cite{bottcher2005spectral}.  For our numerical experiments in Section \ref{numericalresults}, we set $n = 100$ and $\varepsilon = 1/100$.

\subsubsection{Stationary distribution} \label{stationarydist}
Here we consider the problem of finding the stationary distribution of a Markov chain.
Recall that the circular law says that if $\vect{M}$ is an $n \times n$ matrix with i.i.d. mean $0$ variance $1$ entries, then the spectral measure of $n^{-1/2} \vect{M}$ converges to the uniform distribution on the complex unit disk, see for example \cite[Theorem 2.8.1] {tao2012topics}. A related phenomenon is that, in many cases, random directed graphs have complex eigenvalues in a region about the origin. 

To create an example with a small spectral gap, we construct a random barbell graph as follows. Let $\vec{X}$ and $\vec{X}'$ be independent $n \times n$ matrices whose entries are i.i.d. Bernoulli random variables to $1$ and $0$ with probability $p$ and $1-p$, respectively, which represents a random directed graph with possible self-loops. Let $\vect{B}$ be the $2n \times 2n$ block matrix whose $n \times n$ diagonal blocks are $\vect{X}$ and $\vect{X}'$, and whose $n \times n$ off-diagonal blocks are zero, except for the entries $\vect{B}_{n,n+1} = \vect{B}_{n+1,n} = 1$. Let $\vect{P}$ be the $2n \times 2n$ column stochastic matrix $\vect{P} = \vect{B} \vect{D}^{-1}$, where $\vec{D}$ is the diagonal matrix with entries $\vect{D}_{ii} = \sum_{k=1}^n \vect{B}_{ki}$. We construct $\vect{P}$ with $n = 16000$ and $p = 1/1000$ and report the numerical results in Section \ref{numericalresults} as well.

\subsection{Numerical results} \label{numericalresults}
For all examples, we compute the relative error per iteration of the power method (Algorithm \ref{alg1}), the momentum method of \cite{xu2018accelerated}, the deltoid momentum power method (Algorithm \ref{algnew}), and the dynamically accelerated deltoid momentum power method (Algorithm \ref{alg:dymo2}).  
For the deltoid momentum power method, we set $\beta = 4\lambda_2^3/27$ using an oracle value for $\lambda_2$,  see Figure 
\ref{fig:relerr-momentum2}. 
\begin{figure}[h!]
    \centering
    \begin{tabular}{cc}
    \includegraphics[width=0.46\textwidth]{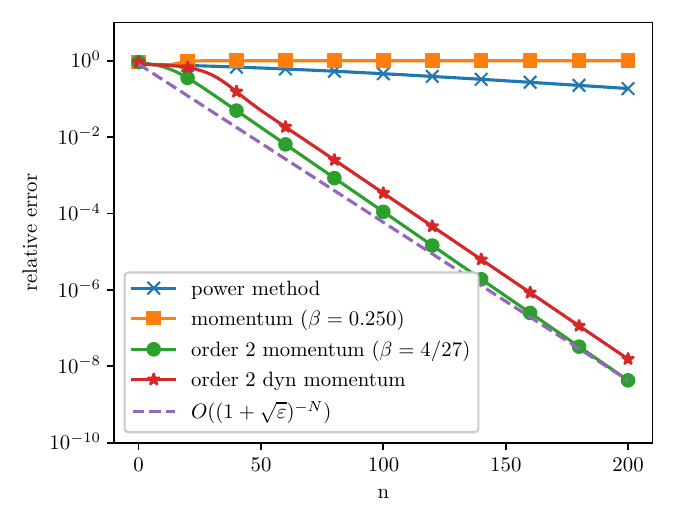} &
    \includegraphics[width=0.46\textwidth]{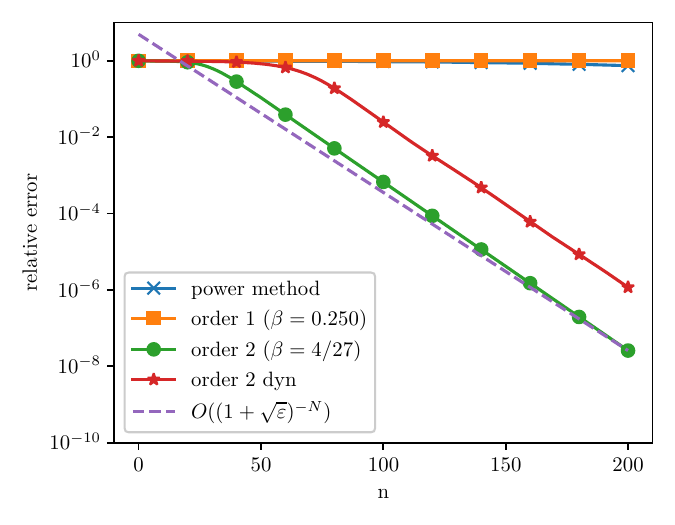} 
    \end{tabular}
    
    \includegraphics[width=0.46\textwidth]{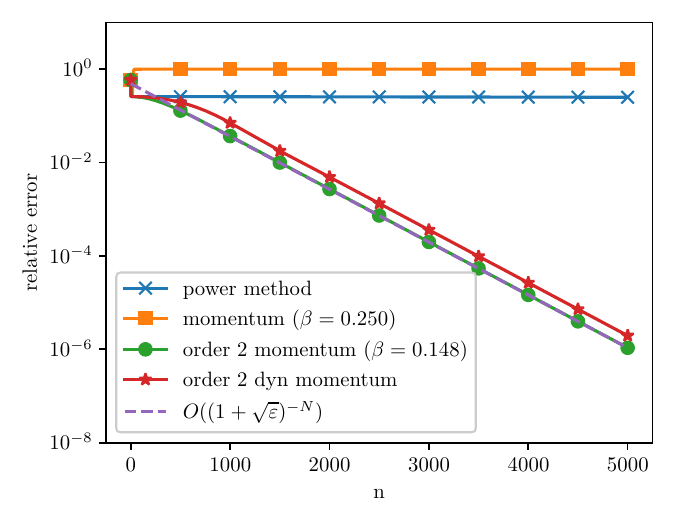} 
    
    \caption{Results for the numerical example of Section \ref{toyexample} (top left), Section \ref{circulantexample} (top right) and Section
    \ref{stationarydist} (bottom), performed with Algorithm \ref{alg1} (power method), the algorithm of \cite{xu2018accelerated} (order 1),  Algorithm \ref{algnew} (order 2), and Algorithm \ref{alg:dymo2} (order 2 dyn).  The dotted line is the asymptotic rate given by Theorem \ref{algoworks}.}
    \label{fig:relerr-momentum2}
\end{figure}

To account for the phase ambiguity of the definition of an eigenvector, the relative error is computed by
$$\text{relative error} = \frac{1}{\| \vect{\phi}_{1}\|_2} \cdot \left\|\frac{\langle \vect{x}, \vect{\phi}_{1}\rangle}{\|\vec{x}\|_2^2}\vect{x} - \vect{\phi}_{1}\right\|_2
=\sin\big(\theta(\vec{x},\vec{\phi}_1)\big),
$$
where $\theta(\vec{x},\vec{\phi}_1) \in [0, \pi)$ is the angle between $\vec{x}$ and $\vec{\phi}_1$. The momentum method presented in \cite{xu2018accelerated} is developed for matrices with real eigenvalues. Here we demonstrate that setting $\beta = \lambda_2^2/4$ (the optimal parameter for symmetric matrices) results in non-convergence.

Figure \ref{fig:relerr-momentum2} illustrates that in both of our examples, the deltoid momentum power method with oracle parameters converges much more quickly than the power method, at a rate closely matched by Theorem \ref{algoworks}.  Moreover, even without a priori knowledge of $\lambda_2$, the dynamic deltoid momentum power method achieves a similar rate of convergence to the deltoid momentum power method.

\section{Proof of main results} 

\subsection{Proof of Theorem \ref{newpolyboundedgrow}} \label{proofthm1}

The proof of Theorem \ref{newpolyboundedgrow} is divided into three parts (Lemma \ref{step1}, Lemma \ref{step2}, and Lemma \ref{step3}), which together establish the theorem.

\begin{lemma} \label{step1}
For all $n \ge 0$, we have
$$
|P_n(z)| \le 1, \quad \text{for all} \quad z \in \Gamma.
$$
\end{lemma}
\begin{proof}[Proof of Lemma \ref{step1}]
When $z \in \mathbb{C}$ is fixed, the polynomial $P_n(z)$ is defined by a homogeneous linear recurrence relation \eqref{newpoly}, which can be solved using the roots of its characteristic polynomial, see for example \cite[Chapter 2]{Elaydi2005}.
In particular, for fixed $z \in \mathbb{C}$, the characteristic polynomial of the recurrence defining $P_n(z)$ is 
$$
p_z(r) = r^3 - \frac{3}{2} z r^2 + \frac{1}{2}.
$$
Recall that the discriminant $\Delta(q)$ of a cubic polynomial
$q(x) = a x^3 + b x^2 + c x + d$ is
$$
\Delta(q) = 18 a b c d - 4 b^3 d  + b^2 c^2  -4 a c^3  - 27 a^2 d^2.
$$
Setting $a = 1$, $b = -(3/2) z$, $c = 0$, and $d = 1/2$ gives
\begin{equation} \label{Deltapz}
\Delta(p_z) = \frac{27}{4}(z^3-1).
\end{equation}
The polynomial $p_z$ has distinct roots if and only if $\Delta(p_z) \not =0$. Thus, it follows from \eqref{Deltapz} that $p_z(r)$ has distinct roots when $z \in \mathbb{C} \setminus \{1,e^{2\pi i/3},e^{4\pi i/3}\}$, and repeated roots when $z \in \{1,e^{2\pi i/3},e^{4\pi i/3}\}$.
Since the general solution to a linear recurrence depends on whether the roots of its characteristic polynomial are distinct or repeated, we consider two cases.

\paragraph{Case 1: Repeated roots}
In this case, we can solve the recurrence directly by inspection:
$$
P_n(e^{k 2\pi i/3}) = e^{n k 2\pi i/3},
$$
for all $n \ge 0$. In particular, $|P_n(e^{k 2\pi i/3})| \le 1$, for $k \in \{0,1,2\}$. 

\paragraph{Case 2: Distinct roots} 
Suppose that $z \in \mathbb{C} \setminus \{1,e^{2\pi i/3},e^{4\pi i/3}\}$ and let $r_1,r_2,r_3$ be the distinct roots of $p_z(r)$. As will become apparent, it suffices to consider the case where $p_z$ has a root of magnitude $1$. Suppose that $r_1 = e^{i t}$ is a root of $p_z(r)$ for $t \in [0,2\pi)$, that is,
$$
p_z(e^{it}) = e^{3it} - \frac{3}{2} z e^{2it} + \frac{1}{2} = 0.
$$
Solving for $z =: \gamma(t)$ gives
\begin{equation} \label{derivegamma}
\gamma(t) = \frac{2}{3}e^{it} + \frac{1}{3}e^{-2it}, \quad \text{for} \quad t \in [0,2\pi],
\end{equation}
which is a parameterization of the deltoid curve $\gamma$ defined in \eqref{gamma}. For fixed $t \in [0,2\pi]$ consider the polynomial
$$
p_{\gamma(t)}(r) =
r^3-\frac{3}{2}  \left(\frac{2 e^{i t}}{3}+\frac{1}{3} e^{-2 i t}\right) r^2+\frac{1}{2}.
$$
By construction $p_{\gamma(t)}$ has root $r_1 = e^{i t}$, so we can factor out $(r-e^{it})$, and use the quadratic formula to deduce that the other roots $r_2$ and $r_3$ of $p_{\gamma(t)}$ are  
$$
r_2 = -\frac{1}{4} e^{-2 i t} \left(-1+\sqrt{1+8 e^{3 i t}}\right) \quad \text{and} \quad
r_3 = -\frac{1}{4} e^{-2 i t} \left(-1-\sqrt{1+8 e^{3 i t}}\right).
$$
Since $|\sqrt{1 + 8e^{3it}}| \leq 3$, we conclude that $|r_2|, |r_3| \leq 1$ by the triangle inequality. As we are considering
the case where the roots $r_1,r_2,r_3$ are distinct, we have 
$$
P_n(\gamma(t)) = c_1 r_1^n + c_2 r_2^n + c_3 r_3^n,
$$
for constants $c_1,c_2,c_3$, see for example \cite[Corollary 2.24]{Elaydi2005}. 
Solving for the constants $c_1,c_2,c_3$ by using the initial conditions $P_0(z)=1$, $P_1(z) = z$, $P_2(z) = z^2$ gives
\begin{eqnarray*}
c_1 &=& \frac{1}{18} \left(8+e^{-3 i t}\right) \\
c_2 &=& \frac{1}{36} e^{-3 i t} \left(2 e^{3 i t} \left(5-\sqrt{1+8 e^{3 i t}}\right)-\sqrt{1+8 e^{3 i t}}-1\right) \\
c_3 &=& \frac{1}{36} e^{-3 i t} \left(2 e^{3 i t} \left(5+\sqrt{1+8 e^{3 i t}}\right)+\sqrt{1+8 e^{3 i t}}-1\right),
\end{eqnarray*}
which can be verified by direct substitution. We 
next
show that $|P_n(\gamma(t))| \le 1$ for all $n \ge 0$. Observe that
$$
|P_n({\gamma(t)})| = \left|c_1 r_1^n + c_2 r_2^n + c_3 r_3^n \right|  \le 1/2 + |c_2 r_2| + |c_3 r_3|,
$$
where final inequality uses the fact $|r_1|,|r_2|,|r_3| \le 1$, and $|c_1| \le 1/2$. It remains to show that
$$
|c_2 r_2| + |c_3 r_3| \le 1/2.
$$
Substituting in the formulas for $c_2,r_2,c_3,$ and $r_3$ and collecting terms gives
$$
|c_2 r_2| + |c_3 r_3| = 
\frac{1}{36} \left(\left| -3 \sqrt{1+8 e^{3 i t}}+4 e^{3 i t}+5\right| +\left| 3 \sqrt{1+8 e^{3 i t}}+4 e^{3 i t}+5\right| \right).
$$
Let $\alpha := \sqrt{1 + 8e^{3it}}$ and note that $4e^{3it} = (\alpha^2 - 1)/2$.  Then,
\begin{eqnarray*}
|c_2 r_2| + |c_3 r_3| &=&
\frac{1}{36} \left(\left| -3 \alpha + \frac{\alpha^2}{2} + \frac{9}{2} \right| + \left| 3 \alpha + \frac{\alpha^2}{2} + \frac{9}{2} \right| \right) \\
&=& \frac{1}{36} \frac{1}{2} \left( \left| (\alpha - 3)^2 \right| + \left| (\alpha + 3)^2 \right| \right)  \\
&=&\frac{1}{72}\big((\alpha - 3)(\overline{\alpha} - 3) + (\alpha + 3)(\overline{\alpha} + 3)\big) = \frac{1}{72} \big(  2\alpha\overline{\alpha} + 18 \big).
\end{eqnarray*}
Since $\alpha \overline{\alpha} = |\alpha|^2 = |1 + 8e^{3it}| \leq 9,$ we conclude that 
$$|c_2 r_2| + |c_3 r_3| \leq \frac{1}{72} \cdot (2 \cdot 9 + 18) = \frac{1}{2}.$$
In combination with Case 1, this gives us
$|P_n(\gamma(t))| \le 1$ for all $t \in [0,2\pi]$. By the maximum modulus principle, it follows that $|P_n(z)| \le 1$ for $z \in \Gamma,$ since $\Gamma$ is the closure of the region enclosed by $\gamma$.
\end{proof}

\begin{lemma} \label{step2}
For all $n \ge 0$ and $\varepsilon > 0$ we have
$$
|P_n(z)| \ge \frac{1}{3} (1+\sqrt{\varepsilon})^n, \quad \text{for} \quad z \in \mathbb{C} : |z| = 1+\varepsilon.
$$
\end{lemma}

\begin{proof}[Proof of Lemma \ref{step2}]
In the proof of Lemma \ref{step1}, we proved that $p_z$ has distinct roots when  
$z \not \in \{1,e^{2\pi i/3},e^{4\pi i/3}\}$; in particular, when $|z|=1+\varepsilon$ for $\varepsilon > 0$, the roots $r_1,r_2,r_3$ of $p_z$ are distinct, and 
$$
P_n(z) = c_1 r_1^n + c_2 r_2^n + c_3 r_3^n,
$$
for some coefficients $c_1,c_2,c_3$ determined by the initial conditions, see for example \cite[Corollary 2.24]{Elaydi2005}. Estimates for these roots and coefficients are derived in 
Appendix
\ref{appendixB}. First, by Lemmas \ref{lemtrig}, 
\ref{extendingbound}, \ref{extendingbound2}, we have
$$
r_1 \ge 1+\sqrt{\varepsilon} \quad \text{for} \quad \varepsilon \ge 0.
$$
Second, by Lemma  \ref{lowerboundc1} we have $c_1 \ge 1/3$. Third, by Lemma \ref{estr2r3c2c3} we have $r_3 > -r_2 > 0$ and $c_3 > -c_2 > 0$ for $\varepsilon > 0$, which implies $c_2 r_2 + c_3 r_3 > 0$. We conclude that
$$
P_n(z) = c_1 r_1^n + c_2 r_2^n + c_3 r_3^n \ge \frac{1}{3}(1+\sqrt{\varepsilon})^n,
$$
as was to be shown.
\end{proof}

\begin{lemma} \label{step3}
Suppose that $z \in \mathbb{C}$ satisfies $|z| = 1+\varepsilon$ for $\varepsilon > 0$. Then,
$$
|P_n(z)| \ge P_n(1+\varepsilon).
$$
\end{lemma}
\begin{proof}[Proof of Lemma \ref{step3}]
The polynomials $P_n$ satisfy the same recurrence as Faber polynomials after appropriate scaling and dilating, see 
\eqref{scaledilatepn} when $\lambda = 3/2$,
and cf. \cite[Proposition 2.1]{he1994zeros}.
Due to satisfying the same recurrence, the polynomials $P_n$ share many properties of the Faber polynomials.  In particular, our choice of initial conditions guarantees that the leading coefficient of $P_n$ is positive and $P_n(0) \geq 0$ for all $n \in \mathbb{N}$.  Our choice of initial conditions also allows one to inductively show for all $n \in \mathbb{N}$ and $k \in \{0, 1, 2\}$ that 
\begin{equation}\label{partialfactor}
P_{3n + k}(z) = C_{3n + k} z^k \prod_{j=1}^n (z^3 - a_{3n+k,j}),
\end{equation}
for some $C_{3n + k} > 0$, and $a_{3n+k,j} \in \mathbb{C}$.  
Taken together, these properties enable the application of the interlacing argument presented in \cite[Theorem 3.1]{he1994zeros} to $P_n$, which allows one to additionally conclude that $a_{3n+k,j} \in (0,1)$.  Figure \ref{fig:Pn-plots} visually illustrates this property of $P_n$.
Now, suppose that $z \in \mathbb{C}$ satisfies $|z| = 1 + \varepsilon$, for $\varepsilon > 0$. Write $z = re^{i\theta}$. By \eqref{partialfactor} we have
\begin{align*}
    \left|P_{3n + k}(re^{i\theta}) \right| \!=\! \left| C_{3n + k} r^ke^{ik\theta} \prod_{j=1}^n (r^3e^{3i\theta} - a_{3n+k,j}) \right| 
    \!=\! \left|C_{3n + k} r^k\right| \prod_{j=1}^n |r^3e^{3i\theta} - a_{3n+k,j}|.
\end{align*}
Since $r$ and $a_{3n+k, j}$ are positive real numbers, $|r^3e^{3i\theta} - a_{3n+k,j}|$ is minimized when $e^{3i\theta} = 1$, which occurs when $\theta \in \{ 0, \frac{2}{3}\pi, \frac{4}{3}\pi\}$, and in particular, when $\theta = 0$.  
\end{proof}

\subsection{Proof of Theorem \ref{polynomialapproxcomplex}} \label{proofthm2}
Our proof strategy is motivated by the proof of Theorem \ref{chebxnfast} due to Sachdeva and Vishnoi \cite{sachdeva2014faster}. We construct a Markov random walk $Y_n$ that is a martingale and satisfies $\mathbb{E}[P_{Y_k}(z)] = z^k$, and then use a concentration argument to establish the result. For context on our proof strategy, we refer the reader to \cite[Pages 21--22]{sachdeva2014faster}.

\begin{proof}[Proof of Theorem \ref{polynomialapproxcomplex}]
We define a Markov chain $\{Y_k\}_{k=0}^\infty$ on $\mathbb{Z}$, which starts at $Y_0 = 0$ and, for $k \ge 1$, and has transition probabilities
$p_{j|i} := \mathbb{P}(Y_k = j | Y_{k-1} =i)$ given by
\begin{equation} \label{keyrandomwalk}
\begin{split}
&p_{1|0} = \frac{1}{2}, \quad p_{-1|0} = \frac{1}{2}, \quad p_{2|1} = \frac{3}{4}, \quad p_{-2|1} = \frac{1}{4},  \quad 
p_{-2|-1} = \frac{3}{4}, \quad p_{2|-1} = \frac{1}{4} , \\
&p_{i+1|i} = \frac{2}{3}, \quad p_{i-2|i} = \frac{1}{3}, \quad p_{-i-1|-i} = \frac{2}{3}, \quad p_{-i+2|-i} = \frac{1}{3}, \quad \text{for} \quad i \ge 2.
\end{split}
\end{equation}
Let $P_n(z)$ be the family of polynomials defined in \eqref{newpoly} for $n \ge 0$. 
We extend the definition of $P_n$ for negative $n$ by symmetry $P_{n}(z) = P_{-n}(z)$ for $n < 0$. Under this extension, we claim that
\begin{equation} \label{zkeq}
\mathbb{E}[P_{Y_k}(z)] = z^k.
\end{equation}
We prove this claim by induction.
The base case $k =0$ holds since $Y_0 = 1$ and $P_0(z) = 1$. Assume that \eqref{zkeq} holds up to $k-1$ for $k \ge 1$. By the tower property of conditional expectation, we have
$$
\mathbb{E}[P_{Y_k}(z)] = \mathbb{E}[ \mathbb{E}[ P_{Y_k}(z) | Y_{k-1}]].
$$
We will show 
$\mathbb{E}[ P_{Y_k}(z) | Y_{k-1}] = z P_{Y_{k-1}}(z)$, 
by which it follows that 
$$
\mathbb{E}[ \mathbb{E}[ P_{Y_k}(z) | Y_{k-1}]] 
=  z \mathbb{E}[P_{Y_{k-1}}(z)] = z^k,
$$
where the final step follows from the inductive hypothesis. 

To prove that $\mathbb{E}[ P_{Y_k}(z) | Y_{k-1}] = z P_{Y_{k-1}}(z)$ we consider three cases: $Y_{k-1} = 0$, $Y_{k-1} = 1$, and $Y_{k-1} \ge 2$. The cases $Y_{k-1} = -1$ and $Y_{k-2} \le -2$ follow by the symmetry of $Y_k$ and $P_n(z)$. First, if $Y_{k-1} = 0$, then by the definition of $Y_k$, we have
$$
\mathbb{E}[ P_{Y_k}(z) | Y_{k-1}=0] = \frac{1}{2} P_{1}(z) + \frac{1}{2} P_{-1}(z) = \frac{1}{2} z + \frac{1}{2} z = z P_0(z).
$$
Second, if $Y_{k-1} = 1$, then, by definition of $Y_k$ and $P_n(z)$ we have
$$
\mathbb{E}[ P_{Y_k}(z) | Y_{k-1}=1] = \frac{3}{4} P_{2}(z) + \frac{1}{4} P_{-2}(z) = \frac{3}{4} z^2 + \frac{1}{4} z^2 = z P_1(z).
$$
Third, assume that $Y_{k-1} \ge 2$. Rearranging the recurrence 
\eqref{newpoly} that defines $P_n(z)$ gives
\begin{equation} \label{keyrearrange}
\frac{2}{3} P_{n+1}(z) + \frac{1}{3} P_{n-2}(z) = z P_n(z), \quad \text{for} \quad n \ge 2.
\end{equation}
Using the definition of $Y_k$ and 
\eqref{keyrearrange} gives
$$
\mathbb{E}[ P_{Y_k}(z) | Y_{k-1} \ge 2] = \frac{2}{3} P_{Y_{k-1}+1}(z) + \frac{1}{3} P_{Y_{k-1}-2}(z) = z P_{Y_{k-1}}(z).
$$
Thus, we have proven  that $\mathbb{E}[ P_{Y_k}(z) | Y_{k-1}] = z P_{Y_{k-1}}(z)$, which 
implies $\mathbb{E}[ P_{Y_k}(z) ] = z^k$.

Next, we show that $Y_k$ concentrates near zero. Note that by definition, $Y_k$ is a martingale  $\mathbb{E}[ Y_{k+1} | Y_k,\ldots,Y_0 ] =  \mathbb{E}[ Y_{k+1} | Y_k ] = Y_k$, for $k \ge 0$. To show that $Y_k$ concentrates, we use Freedman's inequality \cite{freedman1975tail}, also see \cite{tropp2011freedman}. 

\begin{lemma}[Freedman's inequality \cite{freedman1975tail}] \label{freedman}
Let $\{Y_{k}\}_{k=0}^n$ be a martingale for a filtration $\{\mathcal{F}_k \}_{k=0}^n$. Let $X_k := Y_k - Y_{k-1}$ be the martingale differences for $k = 1,\ldots,n$. Define
$$
Z_n = \sum_{k=1}^n \mathbb{E}[X_k^2 | \mathcal{F}_{k-1}].
$$
Assume that the $X_k$ satisfies the one-sided bound: 
$
X_k \le K \text{almost surely for } k=1,\ldots,n.
$
Then, for any $t > 0$ and $\sigma^2 > 0$
$$
\mathbb{P}( \exists n \ge 0 :Y_n \ge t \text{ and } Z_n \le \sigma^2) \le \exp \left( - \frac{t^2}{2( \sigma^2 + K t / 3)} \right).
$$
\end{lemma}

Note that Freedman's inequality can be viewed as a refined version of Azuma's inequality, which also applies in this context, but leads to a weaker bound. Let $X_k = Y_k - Y_{k-1}$ for $k = 1,\ldots,n$. Note that $|X_k| \le 3$. To apply Freedman's inequality, we need a bound or estimate on
$$
Z_n = \sum_{k=1}^n \mathbb{E}[X_k^2 | Y_{k-1}].
$$
We claim that $Z_n \le  5 n/2$, surely, for $k=1,\ldots,n$. Indeed, from \eqref{keyrandomwalk}, it follows that
\begin{equation}
\mathbb{E}[X_k^2 | Y_{k-1} = 0 ] = 1, \quad
\mathbb{E}[X_k^2  | Y_{k-1} = \pm 1 ]= 3 \quad \text{and} \quad
\mathbb{E}[X_k^2 | |Y_{k-1}| \ge 2 ] = 2 .
\end{equation}
By the definition of the transition probabilities \eqref{keyrandomwalk},
we have $Z_1=1$, and if $Z_k =  Z_{k-1}+3$, then $Z_{k+1}= Z_k+2$. Hence, the average value of $Z_k-Z_{k-1}$ is at most $(3+2)/2$, which implies  $Z_n \le  5 n/2$. Applying Freedman's inequality Lemma \ref{freedman} gives
$$
\mathbb{P}( Y_n \ge s) \le \exp \left( - \frac{s^2}{5 n + 2 s} \right).
$$
Using the fact that the bound trivially holds when $s \ge n$ and  setting $s = t\sqrt{n}$ gives 
\begin{equation} \label{concentrate}
\mathbb{P}( Y_n \ge t \sqrt{n}) \le \exp \left( - \frac{t^2}{7} \right).
\end{equation}
We note that a stronger bound could be obtained by adding restrictions to $n$ and $t$.

In the following, we use the fact that $\mathbb{E}[P_{Y_k}(z)] = z^k$ together with the concentration result \eqref{concentrate} to complete the proof. By the definition of expectation,
\begin{equation} \label{znexpansion}
z^n = \sum_{k=-n}^n \mathbb{P}(Y_n =k) P_k(z) = \sum_{k=0}^n \beta_k P_k(z),
\quad \text{for} \quad \beta_k = \mathbb{P}(|Y_n| = k),
\end{equation}
where we used the fact that $P_{n}(z) = P_{-n}(z)$. It follows from \eqref{znexpansion} that
$$
\left| z^n - \sum_{k = 0}^{\lfloor t \sqrt{n} \rfloor} \beta_k P_k(z) \right| = \left| \sum_{|k| > t \sqrt{n}}
\mathbb{P}(Y_n = k) P_k(z) \right|,
$$
where we use the fact that $\mathbb{P}(Y_n = k) =0$ for $|k| > n$.
By Lemma \ref{step1}, we have $|P_k(z)| \le 1$ for $z \in \Gamma$. Hence, by the triangle inequality 
$$
\left| \sum_{|k| > t \sqrt{n}}  \mathbb{P}(Y_n = k) P_k(z) \right| \le \sum_{|k| > t \sqrt{n}} \mathbb{P}(Y_n = k) = \mathbb{P}(|Y_{n}| \ge t \sqrt{n}) 
\quad \text{for} \quad z \in \Gamma.
$$
By symmetry, we can use \eqref{concentrate} to obtain the two-sided bound
$\mathbb{P}(|Y_k| \ge t \sqrt{n}) \le 2 e^{-t^2/7},$ which implies 
$$
\left| z^n - \sum_{k = 0}^{\lfloor t \sqrt{n} \rfloor} \beta_k P_k(z) \right|  \le 2 e^{-t^2/7}, \quad  \text{for} \quad z \in \Gamma,
$$
as was to be shown.
\end{proof}

\subsection{Proof of Theorem \ref{algoworks}}
\label{proofalgoworks}

\begin{proof}[Proof of Theorem \ref{algoworks}]
Let $(\lambda_j, \vec{\phi}_j)$ denote the eigenvalue/eigenvector pairs of $\vec{A}$, where $\{\vec{\phi}_j\}$ form a basis of $\mathbb{C}^n$.   
Let $a_1, \dots, a_n$ be the coefficients in the eigenvector expansion of $\vect{v}_0$:
$$\vect{v}_0 = \sum_{j=1}^n a_j\vect{\phi}_j.$$
Multiplying the update step of Algorithm \ref{algnew} (lines 3 to 5) by $h_k$ gives 
\begin{equation}\label{eq:recurrence_h}
h_k\vect{v}_{k+1} = h_k\vect{A}\vect{x}_{k} = \vect{A}\vect{v}_k - \frac{\beta}{h_{k-1}h_{k-2}} \vect{A} \vect{x}_{k-3} = \vect{A}\vect{v}_k - \frac{\beta}{h_{k-1}h_{k-2}}\vect{v}_{k-2}.
\end{equation}
Define $$\vect{w}_k = \left( \frac{3}{2\lambda_*}\right)^k\left(\prod_{j=0}^{k-1} h_j\right) \vect{v}_k.$$  
By multiplying both sides of \eqref{eq:recurrence_h} by $
(3/(2\lambda_*))^{k+1}
\prod_{j=0}^{k-1} h_j$, substituting $\beta= 4\lambda_*^3/27$, and expressing the resulting equation in terms of $\vect{w}_k$ gives
\begin{equation}\label{eq:recurrence_w2} \vect{w}_{k+1} = \frac{3}{2}\frac{\vect{A}}{\lambda_*}\vect{w}_k - \frac{1}{2} \vect{w}_{k-2}.
\end{equation}
By writing the two initial power iterations on the matrix $\frac{2}{3}\vect{A}$ (run in Algorithm \ref{algnew} line 1) in terms of $\vect{w}_k$, one can verify that
\begin{equation}\label{eq:initial_conditions_w}
\vect{w}_1 = \frac{1}{\lambda_*}\vect{A}\vect{w}_0 \quad \text{and} \quad \vect{w}_2 = \frac{1}{\lambda_*^2}\vect{A}^2\vect{w}_0.
\end{equation} 
Since the initial conditions \eqref{eq:initial_conditions_w} and the recurrence relation \eqref{eq:recurrence_w2} match the definition \eqref{newpoly} of $P_n$ with input $z \to \vect{A}/\lambda_*$, we conclude that 
$$
\vect{w}_N = P_N(\vect{A}/\lambda_*)\vect{v}_0 = \sum_{j=1}^n a_jP_N(\vect{A}/\lambda_*)\vect{\phi}_j = \sum_{j=1}^n P_N(\lambda_j/\lambda_*) a_j\vect{\phi}_j.
$$
By Theorem \ref{newpolyboundedgrow} and the assumption that $\lambda_2/\lambda_*, \dots \lambda_n/\lambda_* \in \Gamma$, we have
\begin{equation} \label{keyboundfromthm}
|P_N(\lambda_1/\lambda_*)| \ge  \frac{1}{3}(1 + \sqrt{|\lambda_1/\lambda_*| - 1})^N
\quad \text{and} \quad |P_N(\lambda_j/\lambda_*)| \leq 1, \quad \text{for} \quad  j \ge 2. 
\end{equation}
Dividing $\vect{w}_N$ by $a_1 P_N(\lambda_1/\lambda_*)$ gives
$$
\frac{\vect{w}_N}{a_1P_n(\lambda_1/\lambda_*)} = \sum_{j=1}^N \frac{a_j P_N(\lambda_j/\lambda_*)}{a_1P_N(\lambda_1/\lambda_*)}\vect{\phi}_j 
= \vect{\phi}_1 + \sum_{j=2}^N \frac{a_j P_N(\lambda_j/\lambda_*)}{a_1P_N(\lambda_1/\lambda_*)}\vect{\phi}_j.
$$
The second term on the right-hand side can be bounded as follows:
\begin{align*}
    \left\| \sum_{j=2}^n \frac{a_j P_N(\lambda_j/\lambda_*)}{a_1P_N(\lambda_1/\lambda_*)}\vect{\phi}_j \right\|_2 &\leq \frac{1}{P_N(\lambda_1/\lambda_*)} \sum_{j=2}^n \left( \left| \frac{a_j}{a_1} \right| \left| P_N(\lambda_j/\lambda_*)\right| \|\vec{\phi}_j\|_2 \right) \\
    &\leq \left(1 + \sqrt{|\lambda_1/\lambda_*| - 1}\right)^{-N} \sum_{j=2}^n \left| \frac{a_j}{a_1} \right|,
\end{align*}
where the first inequality follows from the triangle inequality and the second inequality follows from \eqref{keyboundfromthm}.
We conclude that
$$
\frac{\vect{w}_N}{|a_1P_n(\lambda_1/\lambda_*)|}  = \vec{\phi}_1
 e^{i \theta} + \mathcal{O}\left(  (1+\sqrt{|\lambda_1/\lambda_*| - 1})^{-N} \right),
$$
as $N \to 0$, where the phase $e^{i\theta} := a_1/|a_1|$, and the constant implied in the big-$\mathcal{O}$ notation only depends on $a_1, \dots, a_n$.  Since $\vect{x}_N = \vect{w}_N/\|\vect{w}_N\|_2$ and $\vect{w}_N$ are parallel, this completes the proof.
\end{proof}

\section{Conclusion}  \label{discussion}

In this paper, we introduced a family of polynomials $P_n$, which are closely related to Faber polynomials on the deltoid region $\Gamma$ of the complex plane. We generalized two fundamental results based on Chebyshev polynomials to $\Gamma$. First, we proved that $|P_n| \le 1$ on $\Gamma$ and $|P_n(z)| \ge \frac{1}{3}(1+\sqrt{|z|-1})^n$ outside the unit disk. Second, we proved that $z^n$ is approximately a polynomial of degree $\sim \sqrt{n}$ in $\Gamma$ by relating the coefficients of $z^n$ in the basis $ P_n$ to random walk probabilities. These results generalize classic properties of Chebyshev polynomials and the results of \cite{newman1976approximation,sachdeva2014faster} that apply to the interval $[-1,1]$ to a region of the complex plane.

Additionally, we applied our approximation theory to numerical linear algebra. 
In particular, we generalized the results of \cite{xu2018accelerated} by using a higher-order momentum term to accelerate power iteration for certain non-symmetric matrices. We similarly generalized the dynamic momentum power method of \cite{austin2024dynamically}.

We emphasize that our results are more widely applicable, including applications to solving linear systems and computing functions of matrices such as the matrix exponential, whenever the eigenvalues lie in a deltoid region of the complex plane. While this condition is restrictive, we demonstrated it applies to random walks on directed graphs, which have smaller complex eigenvalues.  The results in this manuscript are based on a specific third-order recurrence, which has the potential to be generalized to families of higher-order recurrence formulas. The analysis of such generalizations, however, may present technical challenges, as roots of higher-order polynomials are not, in general, explicitly available.

\bibliographystyle{plain}
\bibliography{refs}

\begin{thebibliography}{10}

\bibitem{aptekarev2009higher}
AI~Aptekarev, VA~Kalyagin, and EB~Saff.
\newblock Higher-order three-term recurrences and asymptotics of multiple orthogonal polynomials.
\newblock {\em Constructive Approximation}, 30:175--223, 2009.

\bibitem{austin2024dynamically}
Christian Austin, Sara Pollock, and Yunrong Zhu.
\newblock Dynamically accelerating the power iteration with momentum.
\newblock {\em Numerical Linear Algebra with Applications}, 31(6):e2584, 2024.

\bibitem{barletta2025momentum}
Alessandro Barletta, Nicholas Marshall, and Sara Pollock.
\newblock Momentum accelerated power iterations and the restarted lanczos method.
\newblock {\em arXiv preprint arXiv:2511.05364}, 2025.

\bibitem{beckermann2009error}
Bernhard Beckermann and Lothar Reichel.
\newblock Error estimates and evaluation of matrix functions via the {F}aber transform.
\newblock {\em SIAM Journal on Numerical Analysis}, 47(5):3849--3883, 2009.

\bibitem{bottcher2005spectral}
Albrecht B{\"o}ttcher and Sergei~M Grudsky.
\newblock {\em Spectral properties of banded Toeplitz matrices}.
\newblock SIAM, 2005.

\bibitem{curtiss1971faber}
JH~Curtiss.
\newblock Faber polynomials and the {F}aber series.
\newblock {\em The American Mathematical Monthly}, 78(6):577--596, 1971.

\bibitem{eiermann1989semiiterative}
Michael Eiermann.
\newblock On semiiterative methods generated by {F}aber polynomials.
\newblock {\em Numerische Mathematik}, 56(2):139--156, 1989.

\bibitem{eiermann1989hybrid}
Michael Eiermann, X~Li, and RS~Varga.
\newblock On hybrid semi-iterative methods.
\newblock {\em SIAM journal on numerical analysis}, 26(1):152--168, 1989.

\bibitem{eiermann1993zeros}
Michael Eiermann and Richard~S Varga.
\newblock Zeros and local extreme points of {F}aber polynomials associated with hypocycloidal domains.
\newblock {\em Electron. Trans. Numer. Anal}, 1:49--71, 1993.

\bibitem{Elaydi2005}
Saber Elaydi.
\newblock {\em An introduction to difference equations}.
\newblock Undergraduate Texts in Mathematics. Springer, New York, NY, 3 edition, March 2005.

\bibitem{freedman1975tail}
David~A Freedman.
\newblock On tail probabilities for martingales.
\newblock {\em the Annals of Probability}, pages 100--118, 1975.

\bibitem{golub2006arnoldi}
Gene~H Golub and Chen Greif.
\newblock An {A}rnoldi-type algorithm for computing page rank.
\newblock {\em BIT Numerical Mathematics}, 46:759--771, 2006.

\bibitem{golub1961chebyshev}
Gene~H Golub and Richard~S Varga.
\newblock Chebyshev semi-iterative methods, successive overrelaxation iterative methods, and second order richardson iterative methods.
\newblock {\em Numerische Mathematik}, 3(1):157--168, 1961.

\bibitem{he1994zeros}
MX~He and EB~Saff.
\newblock The zeros of {F}aber polynomials for an m-cusped hypocycloid.
\newblock {\em Journal of Approximation Theory}, 78(3):410--432, 1994.

\bibitem{hestenes1952methods}
Magnus~R Hestenes and Eduard Stiefel.
\newblock Methods of conjugate gradients for solving linear systems.
\newblock {\em Journal of research of the National Bureau of Standards}, 49(6):409--436, 1952.

\bibitem{heuveline1996arnoldi}
Vincent Heuveline and Miloud Sadkane.
\newblock Arnoldi-{F}aber method for large non {H}ermitian eigenvalue problems.
\newblock {\em [Research Report] RR-3007, INRIA. 1996. ⟨inria-00073688⟩}, 1996.

\bibitem{lanczos1950iteration}
Cornelius Lanczos.
\newblock An iteration method for the solution of the eigenvalue problem of linear differential and integral operators.
\newblock {\em Journal of research of the National Bureau of Standards}, 45(4):255--282, 1950.

\bibitem{moret2001computation}
Igor Moret and Paolo Novati.
\newblock The computation of functions of matrices by truncated {F}aber series.
\newblock {\em Numerical Functional Analysis and Optimization}, 22(5-6):697--719, 2001.

\bibitem{moret2001interpolatory}
Igor Moret and Paolo Novati.
\newblock An interpolatory approximation of the matrix exponential based on faber polynomials.
\newblock {\em Journal of Computational and Applied Mathematics}, 131(1-2):361--380, 2001.

\bibitem{nachtigal1992hybrid}
No{\"e}l~M Nachtigal, Lothar Reichel, and Lloyd~N Trefethen.
\newblock A hybrid {GMRES} algorithm for nonsymmetric linear systems.
\newblock {\em SIAM Journal on Matrix Analysis and Applications}, 13(3):796--825, 1992.

\bibitem{newman1976approximation}
DJ~Newman and TJ~Rivlin.
\newblock Approximation of monomials by lower degree polynomials.
\newblock {\em aequationes mathematicae}, 14:451--455, 1976.

\bibitem{niethammer1983analysis}
Wilhelm Niethammer and Richard~S Varga.
\newblock The analysis of k-step iterative methods for linear systems from summability theory.
\newblock {\em Numerische Mathematik}, 41(2):177--206, 1983.

\bibitem{nigam2022simple}
Nilima Nigam and Sara Pollock.
\newblock A simple extrapolation method for clustered eigenvalues.
\newblock {\em Numerical Algorithms}, pages 1--29, 2022.

\bibitem{pollock2021extrapolating}
Sara Pollock and L~Ridgway Scott.
\newblock Extrapolating the {A}rnoldi algorithm to improve eigenvector convergence.
\newblock {\em International Journal of Numerical Analysis and Modeling}, 18(5):712--721, 2021.

\bibitem{polyak1964some}
Boris~T Polyak.
\newblock Some methods of speeding up the convergence of iteration methods.
\newblock {\em Ussr computational mathematics and mathematical physics}, 4(5):1--17, 1964.

\bibitem{pozza2019inexact}
Stefano Pozza and Valeria Simoncini.
\newblock Inexact {A}rnoldi residual estimates and decay properties for functions of non-hermitian matrices.
\newblock {\em BIT Numerical Mathematics}, 59(4):969--986, 2019.

\bibitem{saad84}
Y.~Saad.
\newblock Chebyshev acceleration techniques for solving nonsymmetric eigenvalue problems.
\newblock {\em Mathematics of Computation}, 42(166):567–588, 1984).

\bibitem{saad1986gmres}
Youcef Saad and Martin~H Schultz.
\newblock Gmres: A generalized minimal residual algorithm for solving nonsymmetric linear systems.
\newblock {\em SIAM Journal on scientific and statistical computing}, 7(3):856--869, 1986.

\bibitem{sachdeva2014faster}
Sushant Sachdeva and Nisheeth~K Vishnoi.
\newblock Faster algorithms via approximation theory.
\newblock {\em Foundations and Trends{\textregistered} in Theoretical Computer Science}, 9(2):125--210, 2014.

\bibitem{starke1993hybrid}
Gerhard Starke and Richard~S Varga.
\newblock A hybrid {A}rnoldi-{F}aber iterative method for nonsymmetric systems of linear equations.
\newblock {\em Numerische Mathematik}, 64:213--240, 1993.

\bibitem{suetin1998series}
Pavel~K Suetin.
\newblock {\em Series of {F}aber polynomials}, volume~1.
\newblock CRC Press, 1998.

\bibitem{tao2012topics}
Terence Tao.
\newblock {\em Topics in random matrix theory}, volume 132.
\newblock American Mathematical Soc., 2012.

\bibitem{trefethen2022exactness}
Lloyd~N Trefethen.
\newblock Exactness of quadrature formulas.
\newblock {\em Siam Review}, 64(1):132--150, 2022.

\bibitem{tropp2011freedman}
Joel Tropp.
\newblock {Freedman's inequality for matrix martingales}.
\newblock {\em Electronic Communications in Probability}, 16(none):262 -- 270, 2011.

\bibitem{xu2018accelerated}
Peng Xu, Bryan He, Christopher De~Sa, Ioannis Mitliagkas, and Chris Re.
\newblock Accelerated stochastic power iteration.
\newblock In {\em International Conference on Artificial Intelligence and Statistics}, pages 58--67. PMLR, 2018.

\end{thebibliography}

\appendix

\section{Formulas for roots and coefficients}
\label{appendixA}

This appendix contains formulas for the roots of a cubic polynomial \eqref{cas1trig} and formulas for the solution to a three-by-three linear system of equations involving the roots \eqref{cas2trig}. The correctness of these equations can be verified by direct substitution. As the calculations, although elementary, are somewhat lengthy, we omit them for brevity. Additionally, a Mathematica notebook verifying these equations by direct substitution is included in the GitHub repository associated with this paper:
\begin{center}\url{https://github.com/petercowal/higher-order-momentum-power-methods}.
\end{center}

\begin{lemma}
Let $z = (1+\delta)^{1/3}$ for $\delta > 0$. Then, the roots of the polynomial 
\begin{equation}  \label{pzappendix}
p_z(r) = r^3 - \frac{3z}{2}r^2 + \frac{1}{2}
\end{equation}
are
\begin{equation} \label{cas1trig}
\begin{split}
r_1 &= \frac{1}{2} (\delta+1)^{1/3} \left(1 + 2 \cos \left( \frac{2}{3} \arccot \left( \sqrt{\delta} \right) \right) \right)\\
r_2 &= \frac{1}{2} (\delta+1)^{1/3} \left(1 + 2 \cos \left( \frac{2}{3} \arccot \left( \sqrt{\delta} \right) + \frac{2 \pi}{3} 
\right) \right)\\
r_3 &= \frac{1}{2} (\delta+1)^{1/3} \left(1 + 2 \cos \left( \frac{2}{3} \arccot \left( \sqrt{\delta} \right) - \frac{2 \pi}{3}
\right) \right).\\
\end{split}
\end{equation}
\end{lemma}

\begin{lemma}
Let $z = (1+\delta)^{1/3}$ for $\delta > 0$, and
 $r_1,r_2,r_3$ be as defined in
\eqref{cas1trig}. Then, the linear system of equations
\begin{equation} \label{coeffeqappedix}
\begin{split}
c_1 + c_2 + c_3 =& 1 \\
c_1 r_1 + c_2 r_2 + c_3 r_3 =& z \\
c_1 r_1^2 + c_2 r_2^2 + c_3 r_3^2 =& z^2
\end{split}
\end{equation}
has solution
\begin{equation} \label{cas2trig}
\begin{split}
c_1 =& \frac{1}{3} \left(1 + (\delta+1)^{1/2} \sin \left(\frac{1}{3} \arccot\left(\sqrt{\delta}\right)\right)\right) \\[5pt]
c_2 =& \frac{1}{3} \left(1 + (\delta+1)^{1/2} \sin \left( \frac{1}{3} \arccot\left(\sqrt{\delta}\right)
- \frac{2\pi}{3} \right)\right) \\[5pt]
c_3 =& \frac{1}{3} \left(1 + (\delta+1)^{1/2} \sin \left(\frac{1}{3} \arccot\left(\sqrt{\delta}\right)
+ \frac{2\pi}{3}  \right)\right).\\[5pt]
\end{split}
\end{equation}
\end{lemma}

\section{Estimates on roots and coefficients} \label{appendixB}

In this section, we derive estimates on the roots $r_1,r_2,r_3$ of the polynomial $p_z(r)$ defined in \eqref{pzappendix}, and coefficients $c_1,c_2,c_3$ defined by  \eqref{coeffeqappedix}. The key estimate is Lemma \ref{lemtrig} that establishes a lower bound on $r_1$ for $z=1+\varepsilon$, for small $\varepsilon > 0$. We establish both an upper and lower bound to show that our analysis is sharp up to lower-order terms.

\begin{lemma}  \label{lemtrig}
Let $z = 1+\varepsilon$. Then,
$1+ \sqrt{\varepsilon} \le r_1 \le 1 + \sqrt{\varepsilon} + 2\varepsilon$, for
$0 \le \varepsilon \le 1/4$.
\end{lemma}

\begin{proof}[Proof of Lemma \ref{lemtrig}]
Let $z = 1+\varepsilon = (1+\delta)^{1/3}$ for $\delta > 0$. 
We establish upper and lower bounds on $r_1$ by using the formula \eqref{cas1trig}. 
Note that the assumption $0 \le \varepsilon \le 1/4$ implies that $0 \le \delta \le (1+1/4)^3 - 1 < 1$. Recall that alternating series whose terms decrease in magnitude can be bounded above and below by truncating the series at a positive and negative term, respectively; by applying this alternating series bound to the Taylor series, we can deduce that
\begin{eqnarray}
&1 + \frac{\delta}{3} - \frac{\delta^2}{9} \le (1+\delta)^{1/3} \le 1 + \frac{\delta}{3},
\label{est1} \\
&\frac{\pi}{3} - \frac{2}{3} \sqrt{\delta} \le \frac{2}{3} \arccot(\sqrt{\delta}) \le \frac{\pi}{3} - \frac{2}{3}\sqrt{\delta} + \frac{2 \delta^{3/2}}{9},
 \label{est2} \\
&\frac{1}{2} + \frac{\sqrt{3}}{2} x - \frac{x^2}{4} -\frac{x^3}{4 \sqrt{3}} \le \cos(\pi/3-x) \le  \frac{1}{2} + \frac{\sqrt{3}}{2} x,
\label{est3} 
\end{eqnarray}
where the final series alternates every two terms. 
Using \eqref{est1}, \eqref{est2}, and \eqref{est3}, 
to upper and lower bound the formula for $r_1$ in \eqref{cas2trig} gives
\begin{equation} \label{combine1}
\frac{2 \delta^{7/2}}{243 \sqrt{3}}-\frac{11 \delta^{5/2}}{81 \sqrt{3}}+\frac{7 \delta^{3/2}}{27 \sqrt{3}}+\frac{\delta^3}{81}-\frac{4 \delta^2}{27}+\frac{2 \delta}{9}+\frac{\sqrt{\delta}}{\sqrt{3}}+1  \le r_1  \le  
1 + \frac{\sqrt{\delta}}{\sqrt{3}} + \frac{\delta}{3} -\frac{\delta^{5/2}}{9 \sqrt{3}}.
\end{equation}
Since we may assume $0 \le \delta < 1$, we can deduce the simplified bound
\begin{equation} \label{combine1cor}
1 + \frac{\sqrt{\delta}}{\sqrt{3}} \le r_1 \le 1 + \frac{\sqrt{\delta}}{\sqrt{3}} + \frac{\delta}{3}.
\end{equation}
Next, we set $1+\varepsilon = (1+\delta)^{1/3}$, which implies
\begin{equation} \label{eps}
3\varepsilon + 3 \varepsilon^2 + \varepsilon^3 = \delta.
\end{equation}
Combining \eqref{combine1cor} and \eqref{eps} gives
\begin{equation} \label{deltaeps}
1 + \frac{\sqrt{3\varepsilon + 3 \varepsilon^2 + \varepsilon^3}}{\sqrt{3}} \le r_1 \le 1 + \frac{\sqrt{3\varepsilon + 3 \varepsilon^2 + \varepsilon^3}}{\sqrt{3}} + \frac{3\varepsilon + 3 \varepsilon^2 + \varepsilon^3}{3}.
\end{equation}
We note that on both sides of \eqref{deltaeps}, we can factor out 
$\sqrt{3 \varepsilon}$
to write
\begin{equation} \label{factoreps3}
\frac{\sqrt{3\varepsilon + 3 \varepsilon^2 + \varepsilon^3}}{\sqrt{3}} = \sqrt{\varepsilon} \sqrt{1 + \varepsilon + \varepsilon^2/3}.
\end{equation}
Using the inequality $1 \le \sqrt{1+x} \le 1 + x/2,$ can bound \eqref{factoreps3} above and below by
\begin{equation} \label{epsbound}
\sqrt{\varepsilon} \leq \sqrt{\varepsilon} \sqrt{1 + \varepsilon + \varepsilon^2/3} \le \sqrt{\varepsilon} + \frac{1}{2}\varepsilon^{3/2} + \frac{1}{6}\varepsilon^{5/2} \leq \sqrt{\varepsilon} + \frac{1}{4}\varepsilon + \frac{1}{48}\varepsilon,
\end{equation}
where we also use our assumption that $0 \le \varepsilon \le 1/4$.
From \eqref{deltaeps}, \eqref{epsbound}, the fact that $0 \le \varepsilon \le 1/4$, 
we can deduce from the simplified bound $1+ \sqrt{\varepsilon} \le r_1 \le 1 + \sqrt{\varepsilon} + 2\varepsilon$.
\end{proof}

\begin{lemma} \label{extendingbound}
Let $z = 1 + \varepsilon$. Then
$r_1 \ge 1 + \sqrt{\varepsilon},$ for $1 \ge \varepsilon \ge 1/4.$
\end{lemma}
\begin{proof}[Proof of Lemma \ref{extendingbound}]
From  \eqref{cas1trig} it is clear that $r_1$ is an increasing function of $\delta$, which is an increasing function of $\varepsilon$. Hence $r_1 = r_1(\varepsilon)$ is an increasing function of $\varepsilon$.
For $1/4+k/8 \le \varepsilon \le 1/4+(k+1)/8$ the function $r_1(\varepsilon)$ is bounded below by $r_1(1/4+k/8)$ while $1+\sqrt{\varepsilon}$ is bounded above by $1+\sqrt{1/4+(k+1)/8}$. 
One can verify that 
$$
 r_1(1/4+k/8) - 1+\sqrt{1/4 + (k+1)/8} > 0,
\quad \text{for} \quad  k \in \{1,\ldots,7\},
$$
which implies $r_1(\varepsilon) \ge 1+\sqrt{\varepsilon}$ on $[1/4,1]$.
\end{proof}

\begin{lemma} \label{extendingbound2}
Let $z = 1 + \varepsilon$. Then
$r_1 \ge 1 + \sqrt{\varepsilon}$, for $\varepsilon \ge 1$.
\end{lemma}
\begin{proof}
Recall that $1+\varepsilon = (1+\delta)^{1/3}$. 
From  \eqref{cas1trig}, $r_1$ is a product of $(\delta+1)^{1/3}$ and $(1+2\cos \left( \frac{2}{3} \arccot \left( \sqrt{\delta} \right) \right))/2$, which are both increasing functions of $\delta$. Note $\delta= 7$, that corresponds to $\varepsilon =1$. Using the fact that
$\left(1+2\cos \left( \frac{2}{3} \arccot \left( \sqrt{7} \right)\right) \right) \approx
1.47112$ gives
$$
r_1 \ge (1+\delta)^{1/3} \frac{1}{2}\left(1+2\cos \left( \frac{2}{3} \arccot \left( \sqrt{7} \right)\right) \right) \ge (1+\delta)^{1/3} \cdot 1.47.
$$
Using the fact that $\delta = 3\varepsilon + 3 \varepsilon^2 + \varepsilon^3$ gives
$$
r_1 \ge (1+3\varepsilon + 3 \varepsilon^2 + \varepsilon^3)^{1/3}\cdot 1.47 \ge 1 + \sqrt{\varepsilon},
\quad \text{for} \quad \varepsilon \ge 1.
$$
\end{proof}

\begin{lemma} \label{lowerboundc1}
Let $z = 1+\varepsilon$. We have
$c_1 \ge 1/3$, for $\varepsilon \ge 0.$
\end{lemma}
\begin{proof}[Proof of Lemma \ref{lowerboundc1}]
This result is a consequence of the formula for $c_1$ in \eqref{cas2trig}: note that $\arccot(\sqrt{\delta})/3 \in [0,\pi/6]$, which implies that the sine of this quantity is nonnegative. We conclude that
$c_1 \ge 1/3$.
\end{proof}

\begin{lemma}\label{estr2r3c2c3}
Let $z = 1+\varepsilon$. Then,
$r_3 > -r_2 > 0$ and  $c_3 > -c_2 > 0$ for $\varepsilon > 0.$
\end{lemma}

\begin{proof}[Proof of Lemma \ref{estr2r3c2c3}]
This result is a consequence of the formulas for $r_2,r_3$ and $c_2,c_3$ in \eqref{cas1trig} and \eqref{cas2trig}, respectively. Indeed, in order to observe that $r_3 > -r_2 > 0$, we can ignore the leading factor of $(1/2) (1+\delta)^{1/3}$, which is the same for both $r_2$ and $r_3$ and perform the substitution $x = (2/3) \arccot(\sqrt{\delta})$. With this substitution the inequality $r_3 > -r_2 > 0$ is equivalent to the inequality 
$$
1 + 2 \cos\left(x - \frac{2\pi}{3} \right) > - \left(1+ 2\cos\left(x + \frac{2 \pi}{3} \right) \right) > 0, \quad \text{for} \quad x \in [0,\pi/3], 
$$
which is straightforward to verify.  
The inequality $c_3 > -c_2 > 0$ can similarly be established.
\end{proof}

\end{document}